\documentclass[11pt, a4paper]{amsart}
\usepackage{amsfonts,amssymb,amsmath,amsthm,amscd,mathtools,multicol,tikz, tikz-cd,caption,enumerate,mathrsfs,thmtools,cite}
\usepackage{inputenc}
\usepackage[foot]{amsaddr}
\usepackage[pagebackref=true, colorlinks, linkcolor=blue, citecolor=red]{hyperref}
\usepackage{latexsym}
\usepackage{fullpage}
\usepackage{microtype}
\usepackage{subfiles} 

\usepackage{palatino}
\makeatletter
\makeindex 
\newcommand{\be}{\begin{equation}}
\newcommand{\ee}{\end{equation}}
\newcommand{\beano}{\begin{eqn*}} 
	\newcommand{\eeano}{\end{eqnarray*}}
\newcommand{\ba}{\begin{array}}
	\newcommand{\ea}{\end{array}}

\declaretheoremstyle[headfont=\normalfont]{normalhead}

\newtheorem{theorem}{Theorem}[section]
\newtheorem{theoremalph}{Theorem}[section]

\newtheorem{lemma}[theorem]{Lemma}
\newtheorem{corollary}[theorem]{Corollary}

\newtheorem{proposition}[theorem]{Proposition}

\newtheorem{example}[theorem]{Example}

\newcommand{\Sp}{\mathrm{Sp}}
\newcommand{\Aut}{\mathrm{Aut}}
\newcommand{\Inn}{\mathrm{Inn}}
\newcommand{\Or}{\mathrm{O}}
\numberwithin{equation}{section}

\newcommand{\img}{\mathrm{image}}

\begin{document}
\title{Word Images and Their Impostors in Finite Nilpotent Groups}
\author{Dilpreet Kaur}
\email{dilpreetkaur@iitj.ac.in}
\address{Indian Institute of Technology Jodhpur}

\author{Harish Kishnani}
\email{harishkishnani11@gmail.com}
\address{Indian Institute of Science Education and Research, Sector 81, Mohali 140306, India}

\author{Amit Kulshrestha}
\email{amitk@iisermohali.ac.in}
\address{Indian Institute of Science Education and Research, Sector 81, Mohali 140306, India}

\thanks{We are thankful to William Cocke and Anupam Singh for their interest in our work.}
\maketitle

\begin{abstract}
It was shown in \cite{Lubotzky_2014} by Lubotzky that  automorphism invariant subsets of finite simple groups which contain identity are always word images. In this article, we study word maps on finite nilpotent groups and show that for arbitrary finite groups, the number of automorphism invariant subsets containing identity which are not word images, referred to as word image impostors, may be arbitrarily larger than the number of actual word images. In the course of it, we construct a $2$-exhaustive set of word maps on nilpotent groups of class $2$ and demonstrate its minimality in some cases.
\end{abstract}

\section{Introduction}
Let $F_d$ denote the free group on $d$ letters and $w \in F_d$. For a group $G$, let $G^d$ denote the group of $d$-tuples in $G$.
The evaluation of $w$ on $d$-tuples induces a map $\tilde{w} : G^d \to G$. The map $\tilde{w}$ is called the \emph{word map} on $G$ corresponding to the word $w$. 
The image of $\tilde{w}$ is denoted by $w(G)$. 

A subset $A \subseteq G$ is defined to be a \emph{word image candidate} if
\begin{enumerate}[(i).]
\item $1 \in A$, and
\item $A$ is \emph{${\rm Aut}(G)$-invariant}; \emph{i.e.}, if $g \in A$, then $\varphi(g) \in A$ for every automorphism $\varphi$ of $G$.
 \end{enumerate}

All word images are word image candidates. In \cite{Lubotzky_2014}, Lubotzky proved that if $G$ is a finite simple group and $A \subseteq G$ is a word image candidate, then $A = w(G)$ for some $w \in F_d$. In fact, $d = 2$ suffices. His proof heavily uses properties of finite simple groups such as their $3/2$-generation \cite{Guralnick-Kantor_2000}. 
In this paper, we show that if $G$ is not simple, then 
there may exist word image candidates which are not word images. We refer to such word image candidates as \emph{word image impostors}. The groups of our main focus are the finite nilpotent groups.

\begin{theoremalph}\label{TheoremA}
A finite nilpotent group does not contain a word image impostor if and only if it is an abelian group of prime exponent. (Theorem \ref{Lubotzky-for-nilpotent})
\end{theoremalph} 

For a group $G$, a subset $W \subseteq F_d$ is called a $d$-\emph{exhaustive set} for word images on $G$, if for every $v \in F_d$ there exists $w \in W$ such that $v(G) = w(G)$.

For nilpotent groups of class $2$, we exhibit a $2$-exhaustive set in the following theorem. The notation ${\rm exp}(G)$ denotes the exponent of $G$, and $G'$ denotes the commutator subgroup $[G,G]$. Symbols $x,y \in F_2$ are the free generators of $F_2$.

\begin{theoremalph}\label{TheoremB}
Let $G$ be a nilpotent group of class $2$. Let $e = {\rm exp}(G)$, $e' = {\rm exp}(G')$  and $f = {\rm exp}(G/Z(G))$. Then
$$W := \{x^m[x,y^n] \in F_2: m \mid e,  n \mid f \text{ and } n \leq e'\}$$
is a $2$-exhaustive set for word images on $G$.
(Theorem \ref{exhaustive-set-in-nilpotent-class-2})
\end{theoremalph}

Subsequently, we exhibit examples where the set $W$ in this theorem is a minimal $2$-exhaustive set (Example \ref{example-64} and Example \ref{example-p8}). It is evident from Theorem \ref{TheoremB} that if $G$ is a nilpotent group of class $2$ and $w \in F_2$, then $w(G)$ is closed under taking inverses and powers.

It follows from Theorem \ref{TheoremA} that special $p$-groups (see \S\ref{preliminaries}) contain word image impostors. By Theorem \ref{TheoremB},
we have a complete description of word images $w(G); ~w \in F_2$, for such groups. For the subclasses of extraspecial $p$-groups, we make very explicit calculations to show that word image impostors may heavily outnumber word images. 

\begin{theoremalph}
(Theorem \ref{counting-impostors-in-extraspecials})
Let $p$ be a prime and $G$ be an extraspecial-$p$ group.
Then the only words images in $G$ are $\{1\}$,  $Z(G)$ and $G$. Further, if $i_G$ is the number of word image impostors in $G$ then,
\begin{enumerate}[(i).]
\item If $p = 2$ then 
$$i_G = 
\begin{cases}
1, \quad \text{if } G\cong Q_2 \\
5, \quad \text{if } G\ncong Q_2
\end{cases}
$$

\item If $p \neq 2$ then
$$i_G = 
\begin{cases}
1, ~\quad \quad \quad \quad \text{if } ${\rm exp}(G) = p$ \\
2^{p+1}-3, \quad \text{if } {\rm exp}(G) = p^2 \text{ and }
|G| = p^3 \\
2^{p+2}-3, \quad \text{if } {\rm exp}(G) = p^2 \text{ and }
|G| > p^3 \\
\end{cases}
$$
\end{enumerate}
\end{theoremalph}

The organization of the article is as follows. In \S\ref{preliminaries}, we recall basics of special $p$-groups and recollect a result from \cite{Winter_1972} that describes
automorphisms of extraspecial $p$-groups in terms of some linear groups over finite prime fields.

In subsequent sections \S\ref{words-in-class-2-groups} and \S\ref{impostors-in-extraspecials} we prove main results (Theorem A, Theorem B, Theorem C) of the article.
We conclude the article in \S\ref{special-p-using-word-images} with Theorem \ref{special-through-word-images} which establishes that a nonabelian finite group $G$ in which $\{1\}, Z(G)$ and $G$ are the only word images is necessarily a special $p$-group.

\section{Special $p$-groups and a theorem of Winter}\label{preliminaries}
Let $p$ be a prime. A $p$-group is called \emph{special $p$-group} if its center, derived subgroup and Frattini subgroup coincide and all are isomorphic to an elementary abelian $p$-group. Therefore, special $p$-groups are nilpotent groups of nilpotency class $2$.

For a special $p$-group $G$, both the center $S := Z(G)$ and the quotient group $V := \frac{G}{Z(G)}$ are elementary abelian $p$-groups. Thus we can treat $S$ and $V$ as vector spaces over the prime field $GF(p).$ The map $B_G: V \times V \to S$ defined by $B_G(gZ(G), hZ(G)) = [g,h] := ghg^{-1}h^{-1}$, for $gZ(G), hZ(G) \in V$, is a nondegenrate alternating bilinear map. Also, the image of $B_G$ spans $S$ as a vector space over $GF(p)$, as it is equal to the derived subgroup of $G$. It is evident that the image of $B_G$ is same as the image of word $[x,y] := xyx^{-1}y^{-1} \in F_2$ on the group $G$.

Let $p = 2$. The map $q_G: V \to S$ defined by $q_G(gZ(G))=g^2$, for $gZ(G) \in \frac{G}{Z(G)}$, is a quadratic map. Moreover, the polar map associated with the quadratic map $q_G$ is same as the bilinear map $B_G$ defined above. It follows from \cite[Theorem 1.4]{ObedPaper} that the converse of this result is also true. Let $V$ and $S$ be two vector spaces defined over the prime field $GF(2).$ Let $q: V\to S$ be a quadratic map. The group $G= \{ (v,s) ~:~ v\in V, s\in S \}$ with the group operation
$$(v,s) + (v',s') = (v+v', s+s' + c(v,v'))$$
is a special $2$-group. Here, $c \in Z^2(V,S)$ is the $2$-cocycle corresponding to $q$, as in \cite[Prop. 1.2]{ObedPaper}. In fact, this is a one to one correspondance between isomorphism classes of special $2$-groups and isometry classes of quadratic maps defined over the field $GF(2)$. 

Similar result also holds for odd primes. Let $p$ be an odd prime and $G$ be a special $p$-group. From \cite[Ch. 2, Lemma 2.2$(ii)$]{GorensteinBook} and the fact that the derived subgroup of $G$ is elementary abelian, the map $T_G: V \to S$ defined by $T_G(gZ(G))=g^p$, $gZ(G) \in V$, is linear. Conversely, given a pair $(B,T)$, where $B : V \times V \to S$ is a nondegenerate alternating bilinear map and $T : V \to S$ is a linear map, the following proposition provides a construction of a special $p$-group $G$ such that $B = B_G$ and $T = T_G$.

\begin{proposition}\label{from-b-T-to-special}
Let $p$ be an odd prime. Let $V$ and $S$ be two finite dimensional vector spaces over $GF(p).$  Let $\{v_1 , v_2 ,\dots, v_n \}$ and $\{s_1 , s_2 ,\dots, s_m \}$ be bases of $V$ and $S$, respectively, over $GF(p)$. Let $B : V\times V \to S$ be a nondegenerate alternating bilinear map such that ${\rm span}({\rm image}(B)) = S$ and $T : V\to S$ be a linear map. Then, 
$$G = \langle s_i, v_j : s_i^p  = [s_i , v_j] = [s_i, s_l] = 1, [v_j , v_k] = B(v_j, v_k ), v_j^p = T(v_j) ; 1\leq i,l \leq m, 1\leq j, k\leq n\rangle$$
is a special $p$-group, with $B_G = B$ and $T_G = T$.
Here, the notation $s_i, v_j$ is used for both, the generating symbols of the group $G$ as well as the basis vectors of $S$ and $V$.
\end{proposition}

\begin{proof}
It is clear from the presentation of $G$ that ${\rm exp}(G) = p$ or $p^2$. Thus, $G$ is a $p$-group. Again, from the presentation of $G$, we have $S\subseteq Z(G)$ and from the nondegeneracy of $B$ we have $S=Z(G)$. Since $B$ is bilinear, ${\rm span}({\rm image}(B))  = [G,G]$. Now, the Frattini subgroup $\Phi(G) = G^p[G,G] = S$, as $[G,G]=S$ and $G^p=\img(T)\subseteq S$. 
Thus, $Z(G)=[G,G]=\Phi(G)$ and $G$ is a special $p$-group.
\end{proof}

A special $p$-group $G$ is called \emph{extraspecial $p$-group} if $|Z(G)|=p$. For every $n\in \mathbb{N}$, there are two extraspecial $p$-groups, up to isomorphism, of order $p^{2n+1}$. There is no extraspecial $p$-group of order $p^{2n}$.

If $p$ is an odd prime, then one of the two extraspecial $p$-groups of order $p^{2n+1}$ has exponent $p$. The linear map $T$ corresponding to this group is the zero map.  The extraspecial $p$-group corresponding to nonzero linear map has exponent $p^2$.

Winter, in \cite{Winter_1972}, explained the automorphisms of extraspecial $p$-groups in terms of symplectic group $\Sp(V)$, if $p \neq 2$; and orthogonal group $\Or(V,q)$, if $p = 2$. His main theorem is the following.

\begin{theorem}\cite[Th. 1]{Winter_1972} \label{Winter-Theorem}
Let $p$ be a prime, $G$ be an extraspecial $p$-group and
$V = G/Z(G)$. Let $\Aut_{Z(G)}(G)$ be the subgroup of ${\Aut}(G)$ consisting of automorphisms which act trivially on the  $Z(G)$. Let $\Inn(G)$ be the subgroup of $\Aut_{Z(G)}(G)$ consisting of inner automorphisms of $G$. 

\begin{enumerate}[(i).]
\item There exists $\theta \in \Aut(G)$ such that the order of $\theta$ is $p-1$, $\Aut_{Z(G)}(G)\cap \langle \theta \rangle = \{1\}$, restriction of $\theta$ to $Z(G)$ is a surjective power map, and $\Aut(G)=\langle \theta \rangle \Aut_{Z(G)}(G)$.
\item If $p$ is odd, the quotient $\Aut_{Z(G)}(G)/\Inn(G)$ is isomorphic to a subgroup $Q$ of $\Sp(V)$, where
\begin{enumerate}[(a).]
\item $Q = \Sp(V)$, if $\exp(G) = p$.
\item $Q$ is a proper subgroup of $\Sp(V)$, if $\exp(G) = p^2$.
\end{enumerate}
\item If $p = 2$, then $Q = \Or(V,q)$, where $q:V\to GF(2)$ is the quadratic form associated to $G$.
\end{enumerate}
\end{theorem}

\begin{lemma}\label{conjugacy-classes-of-extraspecial-p}
Let $G$ be an extraspecial $p$-group. Let $g \in G \setminus Z(G)$. Then the coset $gZ(G) \subseteq G$ is the conjugacy class of $g$.
\end{lemma}
\begin{proof}
For an arbitrary $h \in G$, it is clear that $[h,g] \in Z(G)$. Thus, $hgh^{-1} \in gZ(G)$ for all $h \in G$. Since $G$ is a $p$-group and $g$ is noncentral, the size of the conjugacy class of $g$ is divisible by $p$. This forces $gZ(G)$ to be the conjugacy class of $G$.
\end{proof}

\section{Words images on nilpotent groups of class $2$}
\label{words-in-class-2-groups}
Throughout this section, $G$ denotes a finite nilpotent group. In some results of this section, we shall impose
an additional restriction on the nilpotency class. 

\begin{lemma} \label{if-nonsurjective-then-in-Frattini}
Let $G$ be a finite $p$-group and $\Phi(G)$ be its Frattini subgroup. Let $w: G^{(d)} \to G$ be a nonsurjective word map. Then $w(G) \subseteq \Phi(G)$.
\end{lemma}
\begin{proof}
Since $w$ is nonsurjective, its image $w(G)$ is equal to the image of a word of the form $x^{pr}c$, where $r \in \mathbb Z$ and $c \in [F_d, F_d]$ (see \cite[Lemma 2.3]{CockeHoChirality}). Thus, 
$w(G) \subseteq G^p[G,G] = \Phi(G)$, where the last equality of holds because $G$ is a $p$-group.
\end{proof}

\begin{theorem}\label{Lubotzky-for-nilpotent}
Let $G$ be a finite nilpotent group. Then $G$ does not contain word image impostors if and only if $G$ is an abelian group of prime exponent.
\end{theorem} 
\begin{proof}
Let $G$ is an abelian $p$-group of exponent $p$. If $A$ is a word image candidate, then $A = \{1\}$ or $G$. In both cases, $A$ is the image of a word map. Thus, $G$ does not contain word image impostors.

For the converse, let $G$ be a nilpotent group which does not contain word image impostors. 
We first assume that $G$ is a $p$-group. If $G$ is either nonabelian or not of the prime exponent, then, $\Phi(G) = G^p[G,G] \neq 1$. Let $A = (G\setminus \Phi(G)) \cup \{1\}$.
Clearly, $A$ is an automorphism invariant proper subset of $G$ and $1 \in A$. We claim that if $w : G^{(d)} \to G$ is a word map then $A \neq w(G)$.
Assume, to the contrary, that there is a word map $w : G^{(d)} \to G$ such that $A  = w(G)$. Then, using Lemma
\ref{if-nonsurjective-then-in-Frattini}, 
$(G\setminus \Phi(G)) \cup \{1\} = A = w(G) \subseteq \Phi(G)$. This is a contradiction. Hence, $G$ is an abelian group of prime exponent.

Finally, suppose that $G$ is an arbitrary finite nilpotent group which does not contain word image impostors. We write $G$ as a direct product of its Sylow subgroups: 
$G=H_{p_1} \times \dots \times H_{p_k}$. Since ${\rm Aut}(G) = {\rm Aut}(H_{p_1}) \times {\rm Aut}(H_{p_2}) \times \cdots \times {\rm Aut}(H_{p_k})$, we conclude that none of the subgroups $H_{p_i}$ contains impostors. By the theorem in the case of $p$-groups, each $H_{p_i}$ is an abelian group of exponent $p_i$. Thus ${\rm exp}(G) = p_1 p_2 \cdots p_k$.
Let $A'$ denote the subset of $G$ consisting of all elements of order $p_1 \dots p_k$ in $G$. Then, it is easy to check that $A = A' \cup \{1\}$ is a word image candidate and it is not the image of a power map if $k \geq 2$.
Since $G$ is abelian, every word image is the image of a power map. Thus, $k = 1$ and the exponent of $G$ is prime.
\end{proof}

We now introduce some notation. For $r$-tuples $I = (i_1, i_2, \cdots, i_r), J = (j_1, j_2, \cdots, j_r) \in \mathbb Z^r$ and an integer $s < r$, we denote,
\begin{align*}
I_s &:= (i_1, i_2, \cdots, i_s), \quad J_s := (j_1, j_2, \cdots, j_s)\\
|I| &:= i_1 + i_2 + \cdots + i_r \\
|J| &:= j_1 + j_2 + \cdots + j_r \\
I.J & := i_1 j_1 + i_2 j_2 + \cdots + i_rj_r \\
w_{I,J} &:= x^{i_1}y^{j_1}x^{i_2}y^{j_2}\dots x^{i_r}y^{j_r}  \in F_2\\
c_{I,J} &:= [x^{i_1},y^{j_1}][x^{i_2},y^{j_2}]\dots [x^{i_r},y^{j_r}]  \in F_2
\end{align*}
Here, $x,y \in F_2$ are its free generators.
\begin{lemma}\label{nilpotent-2 groups-wIJ}
Let $I, J \in \mathbb Z^r$, be such that $|I| = 0 = |J|$. Then, there exist $\tilde{I}, \tilde{J} \in \mathbb Z^{r}$ such that
for all nilpotent groups of class $2$, the words $w_{I,J}$ and $c_{\tilde{I},\tilde{J}}$
have the same image.
\end{lemma}

\begin{proof}
Let $G$ be a nilpotent group of class $2$. We use induction on $r$ to show the existence of $\tilde{I}, \tilde{J} \in \mathbb Z^r$ such that $w_{I,J}$ and $c_{\tilde{I},\tilde{J}}$ have the same image.
If $r  = 1$, then $w_{I,J} = 1 \in F_2$ and $c_{(0),(0)} = 1$.
If $r = 2$, then $\tilde{I} = (i_1, 0), \tilde{J} = (j_1, 0)$ satisfy
$w_{I,J} = c_{\tilde{I},\tilde{J}}$.
For $r > 2$, let $g \in  w_{I,J}(G)$, and $a, b \in G$ be such that
$g = w_{I,J}(a,b)$. Then
$g=  w_{I_{r-2},J_{r-2}}(a,b) a^{i_{r-1}} b^{j_{r-1}} a^{i_r} b^{j_r}$.
Since $|I| = 0 = |J|$, we substitute $i_r = -(i_{r-1} + i_{r-2} + \cdots +i_2 + i_1)$ and
$j_r = -(j_{r-1} + j_{r-2} + \cdots + j_2 + j_1)$ to obtain
$$g =  w_{I_{r-2},J_{r-2}}(a,b) a^{i_{r-1}} b^{j_{r-1}} a^{-(i_{r-1} + i_{r-2} + \cdots + i_2 + i_1)} b^{-(j_{r-1} + j_{r-2} + \cdots + j_2 + j_1)}$$

Substituting $a^{-i_{r-1}}$ by $a^{-i_{r-1}} b^{-j_{r-1}} b^{j_{r-1}}$, we get 
$$g =  w_{I_{r-2},J_{r-2}}(a,b) [a^{i_{r-1}}, b^{j_{r-1}}] b^{j_{r-1}}
a^{-(i_{r-2} + \cdots + i_2 + i_1)} b^{-(j_{r-1} + j_{r-2} + \cdots + j_2 + j_1)}$$
Since $G$ is a $2$-step nilpotent group, $[G,G] \subseteq Z(G)$. Thus,
$[a^{i_{r-1}}, b^{j_{r-1}}]$ is central and we bring it to the beginning of the expression so that
$$g = [a^{i_{r-1}}, b^{j_{r-1}}]  w_{I',J'}(a,b)$$
where 
\begin{align*}
I' &= (i_1, i_2, \cdots, i_{r-2}, -(i_{r-2}+i_{r-3} + \cdots + i_2 + i_1)) \\
J' &= (j_1, j_2, \cdots, j_{r-3}, j_{r-2} + j_{r-1}, -(j_{r-1} + j_{r-2} + \cdots + j_2 + j_1))
\end{align*}
are $(r-1)$-tuples of integers with $|I'| = 0 = |J'|$.
Thus, arguing inductively on $r$ we complete the proof.
\end{proof}

\begin{lemma}\label{powers-of-commutators}
Let $G$ be a nilpotent group of class $2$. For $a,b \in G$, denote $[a,b] := aba^{-1}b^{-1}$. Let $n \in \mathbb Z$. Then,
\begin{enumerate}
\item[(i).] $[a,b]^n = [a^n,b] = [a,b^n]$.
Consequently, if $I, J \in \mathbb Z^r$ then $c_{I,J}(a,b)^n = c_{I,J}(a^n,b)$. 
\item[(ii).] $[a^ib^j,a^kb^l]=[a,b]^{il-jk}, \forall a,b\in G$.
\item[(iii).] $(ab)^n=a^n b^n [b,a]^{\frac{n(n-1)}{2}}$.
\item[(iv).] If $w\in F_2$ is a word and $a \in w(G)$ then $a^{n}\in w(G)$.
\end{enumerate}
\end{lemma}

\begin{proof}
$(i)$. First, let $n = -1$. Since $G$ is a nilpotent group of class $2$, conjugation fixes commutators.
Thus $[a,b]^{-1} = [b,a] = a[b,a]a^{-1} = [a^{-1}, b]$.
This allows us to assume that $n \in \mathbb N$, in which case the result follows from \cite[Ch. 2, Lemma 2.2$(i)$]{GorensteinBook}. 

\noindent $(ii).$ It is easy to check that for nilpotent groups of class $2$,
$[g, h_1 h_2] = [g,h_1][g,h_2]$. Thus $[a^i b^j, a^k b^l] = [a^i,a^k b^l][b^j,a^k b^l] = [a^i, b^l][b^j, a^k]$.
Now using part $(i)$, $[a^i, b^l] = [a,b]^{il}$ and $[b^j, a^k] = [b,a]^{jk} = [a,b]^{-jk}$.
Thus $[a^i b^j, a^k b^l] = [a,b]^{il-jk}$.

\noindent $(iii).$ For the case $n > 0$ we refer to
\cite[Ch. 2, Lemma 2.2$(ii)$]{GorensteinBook}.
When $n = -m < 0$, then $(ab)^n = (b^{-1} a^{-1})^m$
and the result follows from $n > 0$ case after an easy computation.

\noindent $(iv).$ Since an arbitrary word in $w \in F_2$ is automorphic to a word of type
$x^m w_{I,J}$ for suitable $I, J \in \mathbb N^r$ with $|I| = 0 = |J|$ (see \cite[Lemma 2.3]{CockeHoChirality}), by Lemma \ref{nilpotent-2 groups-wIJ} we may assume that
$w = x^m c_{I,J}$. Let $g \in x^m c_{I,J}(G)$. Thus, there exist $a, b \in G$ such that $g=a^mc_{I,J}(a,b)$
for suitable $r$-tuples $I = (i_1, i_2, \cdots, i_r)$ and  $J = (j_1, j_2, \cdots, j_r)$.
Now, $g^n=(a^m)^n c_{I,J}(a,b)^n = (a^n)^m c_{I,J}(a^n,b)$, where the last equality holds due to part $(i)$ of this lemma. Thus $g^n$ is indeed in the image of $x^mc_{I,J}$.
\end{proof}

As a consequence of part $(iv)$ of this lemma we observe that if $G$ is a nilpotent group of class $2$ then for each $w \in F_2$, the word image $w(G)$ is closed under taking inverses.

\begin{lemma}\label{product-of-commutators-nilpotent-class-2}
Let $I, J \in \mathbb Z^r$. Then, for all nilpotent groups of class $2$ the words $c_{I,J}$ and $[x, y^{I.J}]$ have
the same image.
\end{lemma}

\begin{proof}
Let $G$ be a nilpotent group of class $2$. Let $g \in c_{I,J}(G)$ and $a, b \in G$ be such that 
$g = c_{I,J}(a,b) = [a^{i_1}, b^{j_1}] \cdots [a^{i_r}, b^{j_r}] $.
Since $[a^{i_k}, b^{j_k}] \in [G,G] \subseteq Z(G)$ for each $k \in \{1, 2, \cdots, r\}$, 
the order of taking product does not matter and we write
$g = \prod_{k = 1}^r [a^{i_k}, b^{j_k}]$.
For each term $[a^{i_k}, b^{j_k}]$ in the product, we use Lemma \ref{powers-of-commutators}$(i)$ to obtain
$$ [a^{i_k}, b^{j_k}] = [a^{i_{k}}, b]^{j_{k}} = [a,b]^{i_k j_k}$$
Thus $g = \prod_{k = 1}^r [a, b]^{i_{k}j_k} = [a, b]^{I.J} = [a,b^{I.J}]$, where the last equality follows
from Lemma \ref{powers-of-commutators}$(i)$. 
Tracing back this calculation one may show that the image of $[x^{I.J},y]$ is contained in the image of
$c_{I,J}$.
\end{proof}

\begin{lemma}\label{prime-divisors-set}
Let $G$ be a nilpotent group of class $2$ and $w \in F_2$ be a word on $G$. 
Let $e := {\rm exp}(G)$, $e' := {\rm exp}(G')$  and $f := {\rm exp}(G/Z(G))$.
For $r \in \mathbb N$, let $\mathcal P_r$ denote the set of prime divisors of $r$.
Then, there exist $m, n \in \mathbb N$ such that $\mathcal P_m \subseteq \mathcal P_e$, $\mathcal P_n \subseteq \mathcal P_f$,
$n \leq e'$, and the word maps $w$ and $x^m[x,y^n]$ have the same image.
\end{lemma}

\begin{proof}
By \cite[Lemma 2.3]{CockeHoChirality}, Lemma \ref{nilpotent-2 groups-wIJ} and Lemma \ref{product-of-commutators-nilpotent-class-2},
we may assume that
$w=x^m[x,y^n]$ for some $m,n \in \mathbb N$. 
Let $g = w(a,b) = a^m[a,b^n] \in w(G)$.
Suppose, $p \in \mathcal P_m \setminus \mathcal P_e$. Then
${\rm gcd}(p,e) = 1$ and there exists $p' \in \mathbb N$ such that $pp' \equiv 1 \mod e$.
Thus $a^{pp'} = a \in G$. Let $\ell \in \mathbb N$ be such that
$m = p\ell$. Let $w' = x^{\ell}[x,y^n]$.
Then 
$g = a^{p\ell}[a^{pp'},b^n] = (a^{p})^{\ell}[(a^p)^{p'},b^n]
= (a^{p})^{\ell}[(a^p),b^{np'}]$. Thus, $g \in w'(G)$.

Conversely, let $g = w'(a,b) \in G$. Then,
$$g = a^{\ell}[a,b^n] = (a^{pp'})^{\ell}[a^{pp'}, b^n] = (a^{p'})^m[a^{p'},b^{np}],$$
and we conclude that $g \in w(G)$. Therefore, $w(G) = w'(G)$. 
A successive iteration of this process allows us to assume that
$\mathcal P_m \setminus \mathcal P_e = \emptyset$, i.e.
$\mathcal P_m \subseteq \mathcal P_e$.\\

Now, we show that we may also assume that $\mathcal P_n \subseteq \mathcal P_f$. Suppose, $p \in \mathcal P_n \setminus \mathcal P_f$. 
Then
${\rm gcd}(p,f) = 1$ and there exists $p' \in \mathbb N$ such that $pp' \equiv 1 \mod f$.
Thus $b^{pp'}z = b \in G$ for some $z \in Z(G)$. Let $\ell \in \mathbb N$ be such that
$n = p\ell$. Let $g = w(a,b)$. Then $g = a^m[a,b^n] = a^m[a, b^{p\ell}]$. Thus, $g \in w'(G)$, where $w' = x^m[x,y^{\ell}]$.

Conversely, let $g = w'(a,b) \in G$. Then, 
$$g = a^m[a,b^{\ell}] = a^m[a,z^{\ell}b^{pp'\ell}] = a^m[a,(b^{p'})^{n}]
.$$
Thus, $g \in w(G)$, and we conclude that $w(G) = w'(G)$.
A successive iteration of this process allows us to 
assume that $\mathcal P_n \subseteq \mathcal P_f$. \\

Finally, since $[x,y^n] = [x,y]^n$ and $e' = {\rm exp}(G')$, the assumption $n \leq e'$ is natural.
\end{proof}

In the next theorem we claim that the assumptions
$\mathcal P_m \subseteq \mathcal P_e$ and $\mathcal P_n \subseteq \mathcal P_f$ may be strengthened to $m \mid e$
and $n \mid f$, respectively.

\begin{theorem}\label{exhaustive-set-in-nilpotent-class-2}
Let $G$ be a nilpotent group of class $2$. Let $e = {\rm exp}(G)$, $e' = {\rm exp}(G')$  and $f = {\rm exp}(G/Z(G))$. Then
$$W := \{x^m[x,y^n] : m \mid e,  n \mid f \text{ and } n \leq e'\} \subseteq F_2$$
is a $2$-exhaustive set for word images on $G$.
\end{theorem}

\begin{proof}
Let $w \in F_2$. From Lemma \ref{prime-divisors-set},
we may assume that
$w=x^m[x,y^n]$, where $\mathcal P_m \subseteq \mathcal P_e$, 
$\mathcal P_n \subseteq \mathcal P_f$ and $n \leq e'$. 

Suppose, $m \nmid e$. Then, there exists a prime $p$ and integers $r, s, \ell, k \in \mathbb N$ with $r > s$ such that
$m = p^r\ell$, $e = p^sk$ and ${\rm gcd}(p,\ell) = 1 = {\rm gcd}(p, k)$. 
We observe that
$m \equiv p^s \ell \left(p^{r-s} + k\right) \mod e$
and ${\rm gcd}(p^{r-s} + k, e) = 1$. Thus, there exists $t \in \mathbb N$ such that $t(p^{r-s}+k) \equiv 1 \mod e$. \\

Let $w' = x^{{p^s} \ell}[x,y^n]$. 
We claim that $w(G) = w'(G)$. Let $g = w(a,b)$. Then,
\begin{align*}
g = a^m[a,b^n] &= \left(a^{p^{r-s} + k}\right)^{p^s\ell}[a, b^n] \\
&=\left(a^{p^{r-s} + k}\right)^{p^s\ell}[a^{t(p^{r-s} + k)}, b^n] \\
&= \left(a^{p^{r-s} + k}\right)^{p^s\ell}[a^{p^{r-s} + k}, b^{nt}].
\end{align*}
Thus $g \in w'(G)$.\\

Conversely, if $g \in w'(G)$. Then, 
\begin{align*}
g = a^{p^s \ell}[a,b^n] &= a^{t(p^{r-s} + k)p^s \ell}[a^{t(p^{r-s} + k)},b^n] \\
& = a^{tm}[a^t, (b^{p^{r-s}+k})^n].
\end{align*}
Thus, $g \in w(G)$, and the claim follows. A successive iteration of this process allows us to assume that $m \mid e$.
We follow a similar process to show that we may assume that $n \mid f$.
Suppose, $n \nmid f$. Then, there exists a prime $p$ and integers $r, s, \ell, k \in \mathbb N$ with $r > s$ such that
$n = p^r\ell$, $f = p^sk$ and ${\rm gcd}(p,\ell) = 1 = {\rm gcd}(p, k)$.
We observe that
$n \equiv p^s \ell \left(p^{r-s} + k\right) \mod f$
and ${\rm gcd}(p^{r-s} + k, f) = 1$. Thus, there exists $t \in \mathbb N$ such that $t(p^{r-s}+k) \equiv 1 \mod f$. \\

Let $w' = x^m[x,y^{{p^s} \ell}]$. 
We claim that $w(G) = w'(G)$. Let $g = w(a,b)$. Then,
for some $z \in Z(G)$,
\begin{align*}
g = a^m[a,b^n] = a^m[a, (bz)^{p^s \ell \left(p^{r-s} + k\right)}]
= a^m[a, b^{p^s \ell \left(p^{r-s} + k\right)}]
\end{align*}
Thus $g \in w'(G)$.\\

Conversely, if $g \in w'(G)$. Then, 
\begin{align*}
g = a^m[a,b^{{p^s} \ell}]
= a^m[a, b^{p^s \ell t(p^{r-s}+k)}]
= a^m[a, b^{nt}]
\end{align*}
Thus, $g \in w(G)$, and the claim follows. A successive iteration of this process allows us to assume that $n \mid f$. 

These arguments shows that
$W = \{x^m[x,y^n] : m \mid e \text{ and } n \mid f, e \leq e'\}$ is a $2$-exhaustive set for word images on $G$.
\end{proof}

We show that in many cases $W$ is a minimal $2$-exhaustive set.
We pick these examples from the class of special $p$-groups.
In special $p$-groups, $e = p^2$ and $f = p$. Thus, $W = \{1, x, x^p, [x,y], x^p[x,y]\}$ is $2$-exhaustive set for special $p$-groups. We express these words in terms of maps $q, B$ and $T$ associated to $G$ as in \S \ref{preliminaries}.

When $p=2,$ we define the map $q+B : V \times V \to S$ by 
$$(q + B)(gZ(G), hZ(G)) = q(gZ(G)) + B(gZ(G), hZ(G))$$ for $gZ(G), hZ(G) \in V$. 

For odd primes $p$, we define the map $T+B : V \times V \to S$ by 
$$(T+B)(gZ(G), hZ(G)) = T(gZ(G))+ B(gZ(G), hZ(G))$$
for all $gZ(G), hZ(G) \in V$. 

The images of maps $q$ and $q+B$ are same as the images of words $x^2$ and $x^2[x,y]$, respectively, for special $2$-groups. The images of maps $T$ and $T+B$ are same as the images of words $x^p$ and $x^p[x,y]$, respectively, for special $p$-groups, when $p$ is odd.

\begin{example}\label{example-64}
\normalfont 
Let $V$ and $S$ be $3$-dimensional vector spaces over $GF(2)$. Let $q : V \to S$ the quadratic map, which is explicitly defined by the following, for a fixed choice of bases of $V$ and $S$.
$$q(\alpha,\beta,\gamma) = (\alpha^2+\beta^2+\alpha \beta, \alpha^2+\alpha \gamma,\beta\gamma)$$
Let $B : V \times V \to S$ the polar map of $q$. Then $B$ is bilinear, and, for the same choice of bases, is given by
$$B( (\alpha_1, \beta_1, \gamma_1), (\alpha_2, \beta_2, \gamma_2)) = (\alpha_1\beta_2-\alpha_2\beta_1, \alpha_1\gamma_2-\gamma_1\alpha_2, \beta_1\gamma_2-\gamma_1\beta_2)$$

Let $G$ be the special $2$-group associated with $q$. The order of $G$ is $2^6 = 64$. We claim that the images of three maps $q, B$ and  $q+B$ are distinct nontrivial proper subsets of $G$. It is clear from the following table $B$ is surjective. Therefore its image is same as center of the group $G$.
 
\begin{center}
 \begin{tabular}{|c|c|c|}
 \hline
 $v_1= (\alpha_1, \beta_1, \gamma_1)$ &  $v_2=(\alpha_2, \beta_2, \gamma_2)$ & $B(v_1, v_2)$\\
 \hline
 $(\alpha_1, \beta_1, \gamma_1)$ & $(0,0,1)$ & $(0, \alpha_1, \beta_1)$\\
 \hline
 $(0,1,\gamma_1)$ & $(1,0,\gamma_2)$ & $(1, \gamma_1, \gamma_2)$\\
\hline
\end{tabular} 
\end{center}

We claim that $(0,0,1)\notin \img(q).$ If possible, let  $q(\alpha,\beta,z)=(0,0,1)$. The definition of $q$ forces $\beta=\gamma=1$. We check that
 $q(0,1,1)=q(1,1,1)=(1,0,1)$, and conclude that the map $q$ is not surjective. Further, $\img(q)$ is different from $\img(q+B)$, since
$$(0,0,1) = q(0,0,1)+B( (0,0,1), (0,1,0) ) \in \img(q+B)
$$
However, $q+B$ is not surjective as $(1,1,1)\notin \img(q+B)$. This can be easily verified from the following table, with $v_2= (\alpha_2, \beta_2, \gamma_2)$.
 
 \begin{center}
 \begin{tabular}{|c|c|c|c|}
 \hline
 $v_1$ & $q(v_1)+B(v_1, v_2)$ & $v_1$ & $q(v_1)+B(v_1, v_2)$\\
 \hline
 $(0,0,0)$ & $(0,0,0)$ & $(1, 0, 0)$ & $(1+\beta_2, 1+\gamma_2, 0)$\\
 \hline
 $(0,1,0)$ & $(1-\alpha_2,0,\gamma_2)$ & $(0,0,1)$ & $(0, \alpha_2, \beta_2)$\\
 \hline
 $(1,1,0)$ & $(1+\beta_2-\alpha_2,1+\gamma_2,\gamma_2)$ & $(1, 0, 1)$ & $(1+\beta_2, \gamma_2-\alpha_2, \beta_2)$\\
 \hline
 $(0,1,1)$ & $(1-\alpha_2,-\alpha_2,1+\gamma_2-\beta_2)$ & $(1,1,1)$ & $(1+\beta_2-\alpha_2, \gamma_2-\alpha_2, 1+\gamma_2-\beta_2)$\\
 \hline
 \end{tabular} 
 \end{center}
 \end{example}
We have verified using GAP that the group $G$ of this example is the only special $p$-group of order less than $256 = 2^8$ for which all five words in $W$ have distinct images. For groups of order $p^8$, such examples always exist. More explicitly, we have the following:

\begin{example}\label{example-p8}
\normalfont 
Let $V$ and $S$ be $4$-dimensional vector spaces over $GF(p)$. Consider the bilinear map $B: V\times V \to S$  defined by
 \begin{center}
  $B((\alpha_1, \beta_1, \gamma_1, w_1), (\alpha_2, \beta_2, \gamma_2, \delta_2)) = (\alpha_1\beta_2-\alpha_2\beta_1, \alpha_1\gamma_2-\gamma_1\alpha_2, \beta_1\gamma_2-\gamma_1\beta_2, \alpha_1\delta_2-\alpha_2\delta_1)$.
 \end{center}
 If $p = 2,$ then define $q:V\to S$ by
 $q(\alpha,\beta,\gamma,\delta)= (\beta^2+\alpha \beta, \alpha \gamma, \beta \gamma, \alpha \delta)$.
 If $p\neq 2,$ then define $T: V \to S$ by 
 $T(\alpha,\beta,\gamma,\delta)= (\beta,0,0,0)$. We note that $q$ is a quadratic map and $T$ is a linear map.
Let $G$ be the special $p$-group of order $p^8$ associated with $q$ or $(B,T)$, according as if $p = 2$ or $p \neq 2$.

We claim that if $w_1 \neq w_2 \in W$ then $w_1(G) \neq w_2(G)$. To prove the claim, we first notice that if $p = 2$, the images of $B, q$ and $q+B$ are nontrivial proper subsets of $S$; and if $p \neq 2$, then the images of
$B,T$ and $T+B$ are nontrivial proper subsets of $S$.

We show that $B$ is not surjective. In fact, $(0,0,1,1)\notin \img(B)$.
If possible, let $$B((\alpha_1, \beta_1, \gamma_1, \delta_1), (\alpha_2, \beta_2, \gamma_2, \delta_2))=(\alpha_1\beta_2-\alpha_2\beta_1, \alpha_1\gamma_2-\gamma_1\alpha_2, \beta_1\gamma_2-\gamma_1\beta_2, \alpha_1\delta_2-\alpha_2\delta_1)=(0,0,1,1)$$
Since $\alpha_1\delta_2-\alpha_2\delta_1=1$, both $\alpha_1$ and $\alpha_2$ can't be zero simultaneously. If $\alpha_1=0$, then $\alpha_2\neq 0$, $\alpha_1\beta_2-\alpha_2\beta_1=0$ and  $\alpha_1\gamma_2-\gamma_1\alpha_2=0$ force $\beta_1=0$ and $\gamma_1=0$. This, in turn, implies $\beta_1\gamma_2-\gamma_1\beta_2=0,$ contradicting $\beta_1\gamma_2-\gamma_1\beta_2=1.$ The case $\alpha_1 \neq 0$ may be handled similarly.

If $p = 2$, we show that $\img(B) \neq \img(q)$.
Note that $b((0,1,0,0), (0,0,1,0) = (0,0,1,0)$. If possible, let $q(\alpha,\beta,\gamma,\delta)= (\beta^2+\alpha \beta, \alpha \gamma, \beta \gamma, \alpha \delta) =(0,0,1,0)$. Then $\beta=\gamma=1$.
Now, if $\alpha=0$, then $\beta^2+\alpha \beta=1$. If $\alpha=1$, then, $\alpha z=1$. Thus, $q(\alpha,\beta,z,w)\neq (0,0,1,0)$ for all $(\alpha,\beta,z,w)$.

If $p \neq 2$ then we show that 
$\img(B) \neq \img(T)$. Note that $B((0,1,0,0), (0,0,1,0)) = (0,0,1,0)$ and $T(\alpha,\beta,\gamma,\delta)\neq (0,0,1,0)$ for all $(\alpha,\beta,\gamma,\delta)$.

If $p = 2$, we show in the following table, that $\img(q+B)$ is surjective.
 \begin{center}
 \begin{tabular}{|c|c|c|}
 \hline
 $v_1$ &  $v_2$ & $q(v_1)+B(v_1, v_2)$\\
 \hline
 $(1, 0,\gamma_1, \delta_1)$ & $(1,1,\gamma_2,\delta_2)$ & $(1, \gamma_2, \gamma_1, \delta_2)$\\
 \hline
 $(0,1,\gamma_1,\delta_1)$ & $(1,1,\gamma_2,\delta_2)$ & $(0, \gamma_1, \gamma_2, \delta_1)$\\
\hline
 \end{tabular} 
 \end{center}
 
If $p \neq 2$, we show in the following table, that $\img(T+B)$ is surjective.
 \begin{center}
 \begin{tabular}{|c|c|c|}
 \hline
 $v_1$ &  $v_2$ & $T(v_1)+B(v_1, v_2)$\\
 \hline
 $(1, \beta_1, 0,0)$ & $(1,\beta_2,\gamma_2\neq 0,\delta_2)$ & $(\beta_2, \gamma_2, \beta_1\gamma_2, \delta_2)$\\
 \hline
 $(1,\beta_1,1,\delta_1)$ & $(0,\beta_2,0,\delta_2)$ & $(\beta_2+\beta_1, 0, -\beta_2, \delta_2)$\\
\hline
 \end{tabular} 
 \end{center} 
 
For all prime numbers $p$, this proves that $G$ has distinct sets as images of all possible five words in $W$.
\end{example}

\section{Word image impostors in extraspecial $p$-groups}
\label{impostors-in-extraspecials}

Let $G$ be an extraspecial $p$-group. Recall, from Theorem \ref{Winter-Theorem}, that $\Aut_{Z(G)}(G)/\Inn(G)$ is isomorphic to a subgroup $Q$ of the symplectic group $\Sp(V)$.  In fact, for $\varphi \in \Aut_{Z(G)}(G)$, we define $f_{\varphi} : V \to V$ by $f_{\varphi}(gZ(G)) = \varphi(g)Z(G)$. Then, by \cite[(3A), p. 161]{Winter_1972}, $f_{\varphi} \in \Sp(V)$. Further, if $f \in Q \subseteq \Sp(V)$, then by \cite[(3D) (3E), p. 162]{Winter_1972}, there exists $\varphi_f \in \Aut_{Z(G)}(G)$ such that $f_{\varphi_{f}} = f$. We shall examine the action  $\psi : Q \times V \to V$ given by $\psi(f,v) = f(v)$.

\begin{proposition}\label{if-isometric-then-automorphic}
Let $G, V, Q$ and $\psi$ be as above. Let $g, h \in G \setminus Z(G)$ and $v = gZ(G), w = hZ(G) \in V$. If $v$ and $w$ are in the same $\psi$-orbit then $g$ and $h$ are automorphic.
\end{proposition}

\begin{proof}
Suppose ${\rm orbit}_{\psi}(v) = {\rm orbit}_{\psi}(w)$. Then,
$f(v) = w$ for some $f \in Q$, and
$$hZ(G) = w = f(v) = f_{\varphi_f}(v) = {\varphi}_f(g) Z(G).$$
Thus, 
${\varphi}_f(g) = h z^{\ell}$ for some $\ell\in \{0,1,\dots, p-1\}$, where $z$ is a generator of $Z(G)$. Since $h$ and $hz^{\ell}$ are conjugates in $G$ (see Lemma \ref{conjugacy-classes-of-extraspecial-p}), there exists $\rho \in \Inn(G)$ such that $\rho(h) = hz^{\ell} = {\varphi}_f(g)$. 
Hence ${\rho}^{-1}{\varphi}_f (g) = h$, and $g$ and $h$ are automorphic.
\end{proof}

The following corollary is immediate from the above proposition.
\begin{corollary}
Let $G, V, Q$ and $\psi$ be as above. Let $n_o$ be the number of nonzero orbits of the action $\psi$ and $n_c$ be the number of noncentral $\Aut(G)$ components of the group $G$. Then, $n_c \leq n_o$.
\end{corollary}

Rest of the section is divided into two subsections : $p = 2$ and $p \neq 2$.

\subsection{Case $p = 2$} 
Let $q : V \to GF(2)$ be the quadratic form associated to $G$. Then, by Theorem \ref{Winter-Theorem}, $Q$ is the orthogonal group $\Or(V,q)$.

\begin{lemma}\label{Witt-and-Orbit}
Let $G$ be an extraspecial $2$-group and $V = G/Z(G)$. Let $q : V \to GF(2)$ be the quadratic form associated to $G$. Then $v,w \in V \setminus \{0\}$ have the same orbit under the action $\psi : Q \times V \to V$ if and only if $q(v) = q(w)$.
\end{lemma}

\begin{proof}
The lemma follows from Witt Extension Theorem in characteristic $2$ (see \cite[Theorem 8.3]{Elman-Karpenko-Merkurjev}), and the fact that in this characteristic, $Q = \Or(V,q)$.
\end{proof}

We observe that if $g \in G \setminus Z(G)$ and $v = gZ(G) \in V$ then order of $g$ is $2$ (resp. $4$) if and only if $q(v) = 0$ (resp. $q(v) = 1$). We use this observation in the proof of
the following theorem.

\begin{theorem}\label{aut-components-for-char-2}
Let $G$ be an extraspecial $2$-group.
\begin{enumerate}[(i).]
\item Two elements $g, h \in G$ are automorphic if and only if the following holds: (a). $g$ and $h$ have same orders, and (b). $g \in Z(G)$ iff $h \in Z(G)$.
\item Let $n$ be the number of orbits of natural ${\rm Aut}(G)$ action on $G$. Then,
$$
n = \begin{cases}
3, \quad \text{if } G \cong Q_2 \\
4, \quad \text{if } G \ncong Q_2
\end{cases}
$$
Here, $Q_2$ is the quaternion group of order $8$.
\end{enumerate}
\end{theorem}

\begin{proof}
$(i)$. It is clear that if $g \in Z(G)$ then $g$ is automorphic to some $h \in G$ if and only if $g = h$. Now, let $g, h \in G \setminus Z(G)$ and $v,w$ be their respective images in $V$. If $g$ and $h$ are of the same order then $q(v) = q(w)$. By Lemma \ref{Witt-and-Orbit}, $v$ and $w$ are in the same $\psi$-orbit. Now, by Proposition \ref{if-isometric-then-automorphic}, $g$ and $h$ are automorphic.

$(ii)$. It follows from $(i)$ that there are two central orbits. If $G \cong Q_2$ then all elements of $G \setminus Z(G)$ are of order $4$, hence these are in the same orbit by part $(i)$. If $G \ncong Q_2$ then $G \setminus Z(G)$ contains elements of order $2$ and $4$. Thus, by part $(i)$, there are two noncentral orbits in this case.
\end{proof}

\subsection{Case $p \neq 2$} 
Let $G$ be an extraspecial $p$-group and $(B,T)$ be the pair
consisting of an alternating bilinear form $B:V \times V \to GF(p)$ and a linear map $T : V \to GF(p)$ that is associated to $G$. If ${\rm exp}(G) = p$ then $T = 0$.

\begin{lemma}\label{Witt-and-Orbit-Odd-p}
Let $G$ be the extraspecial $p$-group with ${\rm exp}(G) = p$.
Let $V, Q, \psi$ be as in the beginning of this section. Then the action $\psi$ is transitive on $V \setminus \{0\}$.
\end{lemma}

\begin{proof}
The lemma follows from the transitivity of $\Sp(V)$ action on $V \setminus \{0\}$ (see \cite[Theorem 3.3]{Wilson-Book}), and the fact that in odd characteristic, $Q = \Sp(V)$ for ${\rm exp}(G) = p$ case.
\end{proof}

\begin{theorem}\label{aut-components-for-char-p-exp-p}
Let $G$ be the extraspecial $p$-group with ${\rm exp}(G) = p$.
\begin{enumerate}[(i).]
\item Two elements $g, h \in G$ are automorphic if and only if the following holds: (a). $g$ and $h$ have same orders, and (b). $g \in Z(G)$ iff $h \in Z(G)$.
\item The natural ${\rm Aut}(G)$ action on $G$ has three orbits.
\end{enumerate}
\end{theorem}

\begin{proof}
$(i)$. By Theorem \ref{Winter-Theorem}$(i)$ , it is clear that if $g, h \in Z(G) \setminus \{1\}$ then $g$ and $h$ are automorphic.
Now, let $g, h \in G \setminus Z(G)$ and $v,w$ be their respective images in $V$. By Lemma \ref{Witt-and-Orbit-Odd-p}, $v$ and $w$ are in the same $\psi$-orbit. Now, by Proposition \ref{if-isometric-then-automorphic}, $g$ and $h$ are automorphic.

$(ii)$. From $(i)$ it follows that there are two central orbits. Since all elements of $G \setminus Z(G)$ have the same order $p$, they are in the same orbit.
\end{proof}

We now turn our attention to the case of extraspecial $p$-groups $G$ with ${\rm exp}(G) = p^2$, where $p$ is an odd
prime. Let $B: V \times V \to S$ be the alternating nondegenerate bilinear form and $T : V \to S$ be the linear map associated to $T$, as in \S\ref{preliminaries}. Then, $V$ has a basis $\mathcal B = \{v_1, w_1, v_2, w_2, \cdots, v_n, w_n\}$ such that $B(v_i, w_i) = 1$ for $1 \leq i \leq n$, and, 
$B(v_i, w_j) = B(v_i, v_j) = B(w_i, w_j) = 0$ for $i \neq j$, $T(v_1) = 1$ and $T(u) = 0$ for $u \in \mathcal B \setminus \{v_1\}$ (see \cite[Prop. 2.5]{Dilpreet2019}). We refer to such a basis as a \emph{special symplectic basis} for $B$.

\begin{lemma}\label{Witt-and-Orbit-Odd-p-minus}
Let $G$ be the extraspecial-$p$ group with ${\rm exp}(G) = p^2$. 
Let $V, Q, \psi$ be as in the beginning of this section. Let $\mathcal B = \{v_1, w_1, v_2, w_2, \cdots, v_n, w_n\}$ be a special symplectic basis for $B$.
\begin{enumerate}[(i).]
\item Let $v,w \in V \setminus \{0\}$
be two distinct vectors. Then, ${\rm orbit}_{\psi}(v) = {\rm orbit}_{\psi}(w)$ if $T(v)=T(w)$ and either $v,w \notin {\rm ker}(T)$ or $v,w \notin {\rm span}(w_1)$.

\item If $|G| = p^3$, the action $\psi$ has exactly $2p-2$ nonzero distinct orbits. These are represented by the elements of the form $av_1, bw_1$, where $a,b \in GF(p) \setminus \{0\}$.

\item If $|G| > p^3$, the action $\psi$ has exactly $2p-1$ nonzero distinct orbits. These are represented the elements of the form $av_1, bw_1, v_2$, where $a,b \in GF(p) \setminus \{0\}$.

\end{enumerate}
\end{lemma}

\begin{proof}
We first prove $(i)$.
We claim that there exists $v' \in {\rm orbit}_{\psi}(v)$ such that
$v'$ is of the form $a_1v_1+b_1w_1+a_2v_2$, where $a_2 \in \{0, 1\} \subseteq GF(p)$. To see this, let $U := {\rm span}(\mathcal B \setminus \{v_1, w_1\})$. The restriction of $T$ to $U$ is the zero map and the restriction of $B$ to $U \times U$ is a nondegenerate alternating bilinear form. Let $p_U:V \to U$ be the natural projection by suppressing $v_1$ and $w_1$.
If $p_U(v) = 0$ then the claim holds with $a_2 = 0$.
If $p_U(v) \neq 0$, then by the transitivity of $\Sp(U)$ action on $U \setminus \{0\}$ (see \cite[Theorem 3.3]{Wilson-Book}), there exists 
$f \in \Sp(U)$ such that $f(p_U(v)) = v_2$. We extend $f$ to $f' \in \Sp(V)$ by defining $f'(v_1) = v_1$ and $f'(w_1) = w_1$. 
Then $v' := f'(v) \in {\rm orbit}_{\psi}(v)$ is of the form $a_1v_1 + b_1w_1 + v_2$. We use the same argument to assert that there exists
$w' \in {\rm orbit}_{\psi}(v)$ such that
$w'$ is of the form $c_1v_1 + d_1w_1 + c_2v_2$, where $c_2 \in \{0, 1\} \subseteq GF(p)$.

Thus, to start with, we assume that $p_U(v)$ and $p_U(w)$ are either $0$ or $v_2$. Further, by the hypothesis $T(v) = T(w)$ we conclude that $a_1 = c_1$. Now, let us consider the two non-disjoint cases.

\noindent {\bfseries Case 1}. $v,w \notin {\rm ker}(T)$.

In this case we have $a_1\ne 0$. If $a_2=0$, then we define an isometry $f_1$ of $V$ whose matrix with respect to the basis $\mathcal B$ is

$$\left(
\begin{matrix}
    1         & 0 & 0 & \dots & 0 \\
    \alpha_1  & 1 & 0 & \dots & 0 \\
    0         & 0 & 1 & \dots & 0 \\
    \hdotsfor{5} \\
    0         & 0 & 0 & \dots & 1
\end{matrix}\right).
$$
Here $\alpha_1 \in GF(p)$ is such that ${a_1}\alpha_1 \equiv b_1 \mod p$. It is easy to check that $f_1 \in Q$ and $f_1(a_1v_1)=a_1v_1+b_1w_1=v$. Thus, $v$ and $a_1v_1$ are in the same $\psi$-orbit. 

If $a_2 =1$ then we define an isometry $f_2$ of $V$ whose matrix with respect to the basis $\mathcal B$ is

$$\left(
\begin{matrix}
    1        & 0 & 0        & 0   & \dots & 0 & 0 \\
    0        & 1 & 0        & -1  & \dots & 0 & 0 \\
    \beta_1  & 0 & \beta_1  & 0   & \dots & 0 & 0 \\
    0        & 0 & 0        & a_1 & \dots & 0 & 0 \\
    \hdotsfor{7}                              \\
    0        & 0 & 0        & 0   & \dots & 1 & 0 \\
    0        & 0 & 0        & 0   & \dots & 0 & 1
\end{matrix}\right).
$$
Here $\beta_1$ is such that ${a_1}\beta_1 \equiv 1 \mod p$. Again, it is easy to check that $f_2 \in Q$ and $f_1(f_2(a_1v_1))=f_1(a_1v_1+v_2)=a_1v_1+b_1w_1+v_2$. Since $a_2\in \{0,1\}$, we conclude that $v$ and $a_1v_1$ are in the same $\psi$-orbit in this case.

Replacing $v$ by $w$ in the above argument we conclude that $w$ and $a_1v_1$ are in the same $\psi$-orbit. Thus ${\rm orbit}_{\psi}(v) = 
{\rm orbit}_{\psi}(w)$. \\

\noindent{\bfseries Case 2}. $v,w \notin {\rm span}(w_1)$.

The case $1$ allows us to assume that $v,w \in {\rm ker}(T)$. Thus, $a_1 = c_1 = 0$.
Further, since $v,w \notin {\rm span}(w_1)$, we have $a_2 = c_2 = 1$. We define an isometry $f_3$ of $V$ whose matrix with respect to the basis $\mathcal B$ is

$$\left(
\begin{matrix}
    1        & 0 & 0      & 0 & 0 & \dots & 0   \\
    0        & 1 & b_1    & 0 & 0 & \dots & 0   \\
    0        & 0 & 1      & 0 & 0 & \dots & 0   \\
    b_1      & 0 & 0      & 1 & 0 & \dots & 0   \\
    0        & 0 & 0      & 0 & 1 & \dots & 0   \\
    \hdotsfor{5}                            \\
    0        & 0 & 0      & 0 & 0 & \dots & 1
\end{matrix}\right).
$$
Again, $f_3 \in Q$ and $f_3(v_2)=b_1w_1+v_2=v$. Similarly, $w$ and $v_2$ are in the same $\psi$-orbit. Thus ${\rm orbit}_{\psi}(v) = 
{\rm orbit}_{\psi}(w)$.

Now we prove $(ii)$ and $(iii)$. Let $v \in V \setminus\{0\}$. As in the proof of $(i)$, we may assume that
$v = a_1v_1 + b_1w_1 + a_2 v_2$.

If $v\notin {\rm ker}(T)$ then, again by part $(i)$, $v\in {\rm orbit}_{\psi}(a_1v_1)$. Since $T\circ f=T,\forall f\in Q$ and $T(\alpha v_1)\neq T(\beta v_1)$ if $\alpha \neq \beta$, the orbits ${\rm orbit}_{\psi}(a_1v_1), a_1\in GF(p)\setminus \{0\}$ are all distinct.

If $v \in {\rm ker}(T)$, then $a_1 = 0$. Hence, $v = b_1w_1 + a_2 v_2$. If $a_2 = 0$, then $v= b_1w_1$. By \cite[(4A), p. 164]{Winter_1972}, we have $f(w_1) = w_1, \forall f\in Q$. Thus the orbits ${\rm orbit}_{\psi}(b_1w_1)$ are all singleton.
If $a_2 \neq 0$ then $v = b_1w_1 + a_2v_2 \notin {\rm span}(w_1)$ and $|G| > p^3$. In this case by part $(i)$, $v \in {\rm orbit}_{\psi}(v_2)$. Since, $0 = T(v_2) \neq T(a_1v_1) = a_1$ for $a_1 \neq 0$, the orbit ${\rm orbit}_{\psi}(v_2)$ is distinct from the orbits ${\rm orbit}_{\psi}(a_1v_1)$.

Thus, the orbits of $\psi$ are as asserted in 
$(ii)$ and $(iii)$.
\end{proof}

\begin{theorem}\label{aut-components-for-char-p-exp-p-square}
Let $G$ be the extraspecial $p$-group with ${\rm exp}(G)=p^2$.
\begin{enumerate}[(i).]
\item Let $V, B, T, \psi$ be as in lemma \ref{Witt-and-Orbit-Odd-p-minus} and $\mathcal B = \{v_1, w_1, v_2, w_2, \cdots, v_n, w_n\}$ be the special symplectic basis for $B$. Let $g,h \in G$ be such that $gZ(G), hZ(G) \notin {\rm span}(w_1)\setminus\{0\} \subseteq V$. Two elements $g, h \in G$ are automorphic if and
only if the following holds: (a). $g$ and $h$ have same orders, and (b). $g \in Z(G)$ iff $h \in Z(G)$.

\item Let $n$ be the number of orbits of natural ${\rm Aut}(G)$ action on $G$. Then,
$$
n = \begin{cases}
p+2, \quad \text{if } |G| = p^3 \\
p+3, \quad \text{if } |G| > p^3
\end{cases}
$$
\end{enumerate}
\end{theorem}

\begin{proof}
$(i)$. Let $g,h \in G$ be the elements of the same order which are either both central or both noncentral.
By Theorem \ref{Winter-Theorem}$(i)$ , it is clear that if $g, h \in Z(G)$ then $g$ and $h$ are automorphic. Now suppose that $g, h \in G \setminus Z(G)$. Let $v, w$ be their respective images in $V$. Since $g$ and $h$ have same orders, $v \in {\rm ker}(T)$ iff $w\in {\rm ker}(T)$. 

Suppose $v,w \in {\rm ker}(T)$. As $v, w \notin {\rm span}(w_1)$, we conclude from Lemma \ref{Witt-and-Orbit-Odd-p-minus}$(i)$ that $v$ and $w$ are in the same $\psi$-orbit. 
Thus, by Proposition \ref{if-isometric-then-automorphic}, $g$ and $h$ are automorphic.

Suppose $v,w \notin {\rm ker}(T)$. Then $T(v) = T(\alpha v_1) = \alpha$ and $T(w) = T(\beta v_1) = \beta$ for some nonzero
$\alpha, \beta \in GF(p)$. As $v, w \notin {\rm span}(w_1)$, from Lemma \ref{Witt-and-Orbit-Odd-p-minus}$(i)$, $v$ and 
$\alpha v_1$ are in the same orbit, and $w$ and $\beta v_1$ are in the same $\psi$-orbit. If $v_1 = g_1 Z(G) \in V$, then by Proposition \ref{if-isometric-then-automorphic}, $g$ and $g_1^{\alpha}$ are automorphic. Similarly, $h$ and $g_1^{\beta}$ are automorphic. Now, by \cite[(3B), p. 161]{Winter_1972}, $g_1^{\alpha}$ and $g_1^{\beta}$ are automorphic. This shows that $g$ and $h$ are automorphic.

$(ii)$. By Theorem \ref{Winter-Theorem}$(i)$, the ${\rm Aut}(G)$ action has two central orbits. Let $g \in G\setminus Z(G)$ be such that $gZ(G) = w_1$. By \cite[Corollary 1]{Winter_1972}, $gZ(G)$ is an ${\rm Aut}(G)$-invariant subset. Let $\varphi \in {\rm Aut}(G)$. Then, $\varphi(g)=gh$ for some $h \in Z(G)$, and for
each $\alpha \in GF(p) \setminus \{0\}$, $\alpha w_1 =   g^{\alpha}Z(G) \in V$. 
If $z \in Z(G)$ then $\varphi (g^{\alpha}z) = g^{\alpha}{h}^{\alpha}\varphi(z) \in g^{\alpha} Z(G)$.

Thus, for each $\alpha \in GF(p)\setminus \{0\}$, 
$\alpha w_1 \in V$ corresponds to a noncentral 
${\rm Aut}(G)$-invariant subset of $G$. By Lemma \ref{conjugacy-classes-of-extraspecial-p}, this ${\rm Aut}(G)$-invariant subset is an orbit of ${\rm Aut}(G)$ action. 

If $|G|=p^3$ then, by part $(i)$, the elements $g$ in $G \setminus Z(G)$ such that $gZ(G) \notin {\rm span}(w_1) \subseteq V$ are in the same ${\rm Aut}(G)$ orbit. Thus, the total number of ${\rm Aut}(G)$ orbits in this case is $2$ (central orbits) + $p-1$ (corresponding to each $\alpha w_1$) + $1$ (corresponding to $gZ(G) \notin {\rm span}(w_1$)) = $p+2$.

If $|G| > p^3$ then, by part $(i)$, the elements $g$ in $G \setminus Z(G)$ such that $gZ(G) \notin {\rm span}(w_1) \subseteq V$ split into two ${\rm Aut}(G)$ orbits. Thus, the total number of ${\rm Aut}(G)$ orbits in this case is $2$ (central orbits) + $p-1$ (corresponding to each $\alpha w_1$) + $2$ (corresponding to $gZ(G) \notin {\rm span}(w_1$)) = $p+3$.
\end{proof}

Now onward, $p$ denotes a prime without a restriction of being even or odd. The following theorem gives a complete description of word images for extraspecial-$p$ groups.

\begin{theorem}\label{three-word-images}
Let $G$ be an extraspecial-$p$ group. Then, the only word images of $G$ are $\{1\}$, $Z(G)$ and $G$. 
\end{theorem}

\begin{proof}
The images $\{1\}$ and $G$ are the obvious ones. Suppose $w : G^{(d)} \rightarrow G$ is a word map with $w(G) \neq \{1\}, G$. There are indeed such words, {\it e.g.} the commutator word $w = xyx^{-1}y^{-1}$. By Lemma \ref{if-nonsurjective-then-in-Frattini},  $w(G) \subseteq \Phi(G) = Z(G)$. The last equality holds because $G$ is an extraspecial-$p$ group. Since $w(G)$ is invariant under automorphisms and,  by Theorem \ref{Winter-Theorem}$(i)$ , all elements of $Z(G) \setminus \{1\}$ are automorphic, $w(G) = Z(G)$.
\end{proof}

\begin{theorem}\label{counting-impostors-in-extraspecials}
Let $G$ be an extraspecial-$p$ group and $i_G$ be the number of word image impostors in $G$.
\begin{enumerate}[(i).]
\item If $p = 2$ then 
$$i_G = 
\begin{cases}
1, \quad \text{if } G\cong Q_2 \\
5, \quad \text{if } G\ncong Q_2
\end{cases}
$$

\item If $p \neq 2$ then
$$i_G = 
\begin{cases}
1, ~\quad \quad \quad \quad \text{if } ${\rm exp}(G) = p$ \\
2^{p+1}-3, \quad \text{if } {\rm exp}(G) = p^2 \text{ and }
|G| = p^3 \\
2^{p+2}-3, \quad \text{if } {\rm exp}(G) = p^2 \text{ and }
|G| > p^3 \\
\end{cases}
$$
\end{enumerate}
\end{theorem}

\begin{proof}
If the number of orbits of $G$ under the natural ${\rm Aut}(G)$ action is $n$, then the number of word image candidates is $2^{n-1}$. By Theorem \ref{three-word-images}, only $3$ of these are word images. Thus $i_G = 2^{n-1} -3$. 

The result follows directly by substituting the value of $n$ from Theorem \ref{aut-components-for-char-2},
Theorem \ref{aut-components-for-char-p-exp-p} and Theorem \ref{aut-components-for-char-p-exp-p-square}.
\end{proof}

It is evident from above theorem that if $p$ is odd and ${\rm exp}(G) = p^2$ then the number of word image impostors grows exponentially with $p$.

\section{Detecting special $p$-groups through word images}
\label{special-p-using-word-images}
We end this article with an observation on the converse of the Theorem \ref{three-word-images}. The converse is not true. The direct products of extraspecial $p$-groups with themselves are counterexamples. Interestingly, a weaker form of the converse is true. To formulate it we first record a lemma.

\begin{lemma}{\label{word-image}}
The commutator subgroup of a group $G$ is always a word image. Moreover, if $G$ is a $p$-group then $\Phi(G)$ is also a word image.
\end{lemma}

\begin{proof}
Let $g \in [G,G]$. Let $\ell_g$ be the minimal number of commutators required to write $g$ as a product of commutators. Let $l_G=\underset{g\in [G,G]}\max(l_g)$.
Then, for $w = [x_1, x_2] \cdots [x_{2\ell_G-1}, x_{2\ell_G}] \in \mathbb{F}_{2\ell_G}$, we indeed have
$w(G) = [G,G]$.

Now, if $G$ is a $p$-group then $\Phi(G)=G^p[G,G] = w'(G)$,
where $w' = x_{2\ell_G+1}^p w$. Thus $\Phi(G)$ is a word image. 
\end{proof}

\begin{theorem}\label{special-through-word-images}
Let $G$ be a nonabelian finite group such that the only  word images of $G$ are $\{1\}$, $Z(G)$ and $G$. Then, $G$ is a special $p$-group for some prime $p$.
\end{theorem}

\begin{proof}
We first observe that $\{1\}, Z(G)$ and $G$ are distinct.
Since $G$ is nonabelian, $Z(G) \neq G$.
If $Z(G) = \{1\}$ then, by hypothesis, $G$ does not have a nontrivial proper word image. Denote the $n^{\rm th}$ power word $x^n$ by $w_n$. If $p$ is a prime divisor of ${\rm exp}(G)$ then $w_p(G)$ is a proper subset of $G$. 
Thus $w_p(G) = \{1\}$. This shows that ${\rm exp}(G) = p$ and $G$ is a $p$-group. This is a contradiction to $Z(G) = \{1\}$.

We now show that, under the hypothesis of the theorem, $G$ is a $p$-group. If we assume the contrary, then there exist distinct primes $p$ and $p'$ dividing 
$\exp(G)$. 
Since the word images $w_p(G)$ and $w_{p'}(G)$ are proper,
$w_p(G)=w_{p'}(G)=Z(G)$.
Consequently, for every $g \in G$ there exists $h \in G$ such that $g^p = h^{p'}$. Thus, for all $g \in G$,
$$g^{p({\rm exp}(G)/p')} = h^{p'({\rm exp}(G)/p')} = h^{{\rm exp}(G)}=1$$
and ${\rm exp}(G)$ divides $p({\rm exp}(G)/p')$. This holds iff $p = p'$. This is a contradiction.

We now show that ${\rm exp}(G)=p$ or $p^2$. If ${\rm exp}(G)>p^2$, then by the hypothesis, $w_p(G)=w_{p^2}(G)=Z(G)$.
Along the lines of the argument in the previous paragraph, it is easy to arrive at the contradiction that ${\rm exp}(G)$ divides ${\rm exp}(G)/p$.

By Lemma \ref{word-image}, $[G,G], \Phi (G)$ are nontrivial proper word images and by the hypothesis $Z(G)$ is a nontrivial proper word image. Thus, $[G,G] = \Phi (G) = Z(G)$.

Thus far we have shown that $G$ is a $p$-group of exponent at most $p^2$ with $[G,G] = \Phi (G) = Z(G)$. To prove that $G$ is special $p$-group it remains to be shown that ${\rm exp}(\Phi(G)) = p$. If ${\rm exp}(G) = p$, then ${\rm exp}(\Phi(G)) = p$ is a triviality. If ${\rm exp}(G) = p^2$, then $w_p(G), \Phi (G)$, $Z(G)$ are nontrivial proper word images, and hence these are equal. Thus 
$w_p(\Phi(G)) = w_p(w_p(G)) = w_{p^2}(G) = \{1\}$. This show that ${\rm exp}(\Phi(G)) = p$ and we conclude that $G$ is a special $p$-group.
\end{proof}

\bibliographystyle{amsalpha}
\bibliography{word-maps}

\end{document}